\newcommand{\Hor}{\mathcal{H}}
\newcommand{\V}{\mathcal{V}}
\newcommand{\J}{\mathcal{J}}

\newcommand{\df}{{\rm d}}

\newcommand{\R}{\mathds R}

\newcommand{\spl}{\mathrm{sp}}

\newcommand{\Ddt}{\tfrac{\mathrm D}{\mathrm dt}}
\newcommand{\ddt}{\tfrac{\mathrm d}{\mathrm dt}}
\newcommand{\Spl}{\mathrm{Sp}}
\newcommand{\GL}{\mathrm{GL}}
\newcommand{\Dim}{\mathrm{dim}}
\newcommand{\Bsym}{\mathrm{B_{\mathrm{sym}}}}
\newcommand{\segnatura}{\mathrm{sign}}

\newcommand{\Ker}{\mathrm{Ker}}

\documentclass{amsart}

\usepackage[all]{xy}

\usepackage{dsfont}

\numberwithin{equation}{section}

\renewcommand{\contentsline}[3]{\csname new#1\endcsname{#2}{#3}}
\newcommand{\newchapter}[2]{\bigskip\hbox to \hsize{\vbox{\advance\hsize by -.5cm\baselineskip=12pt\parfillskip=0pt\leftskip=2cm\noindent\hskip -2cm #1\leaders\hbox{.}\hfil\hfil\par}$\,$#2\hfil}}
\newcommand{\newsection}[2]{\medskip\hbox to \hsize{\vbox{\advance\hsize by -.5cm\baselineskip=12pt\parfillskip=0pt\leftskip=2.5cm\noindent\hskip -2cm #1\leaders\hbox{.}\hfil\hfil\par}$\,$#2\hfil}}
\newcommand{\newsubsection}[2]{\medskip\hbox to \hsize{\vbox{\advance\hsize by -.5cm\baselineskip=12pt\parfillskip=0pt\leftskip=3.5cm\noindent\hskip -2cm #1\leaders\hbox{.}\hfil\hfil\par}$\,$#2\hfil}}

\title{Maslov index in semi-Riemannian submersions}
\author[E. Caponio]{Erasmo Caponio}
\address{Dipartimento di Matematica,\hfill\break\indent
Politecnico di Bari, \hfill\break\indent Via Orabona 4,
70125, Bari, Italy}
\email{caponio@poliba.it}

\author[M. A. Javaloyes]{Miguel Angel Javaloyes}
\address{Departamento de Geometr\'{\i}a y Topolog\'{\i}a.\hfill\break\indent
 Facultad de Ciencias, Universidad de Granada.\hfill\break\indent
 Campus Fuentenueva s/n, 18071 Granada, Spain}
 \email{ma.javaloyes@gmail.es}

\author[P.\ Piccione]{Paolo Piccione}
\address{Departamento de Matem\'atica,\hfill\break\indent
Universidade de S\~ao Paulo, \hfill\break\indent Rua do Mat\~ao
1010,\hfill\break\indent CEP 05508-900, S\~ao Paulo, SP, Brazil}
\email{piccione@ime.usp.br}
\urladdr{http://www.ime.usp.br/\~{}piccione}
\curraddr{Department of Mathematics, \hfill\break\indent
University of Murcia, Campus de Espinardo\hfill\break\indent
30100 Espinardo, Murcia, \hfill\break\indent Spain}

\thanks{First author is supported by M.I.U.R. Research project PRIN07 "Metodi Variazionali e  
Topologici nello Studio di Fenomeni Nonlineari".
Second author was partially supported by Regional J.
Andaluc\'{\i}a Grant P06-FQM-01951, by Fundaci\'on S\'eneca project 04540/GERM/06   and by Spanish MEC Grant MTM2009-10418.
The third author is sponsored by Capes, Brazil, grant BEX 1509/08-0.}

\subjclass[2000]{53C50}

\keywords{semi-Riemannian submersions, O'Neill tensors, Maslov index}

\date{January 8th, 2010}

\begin{document}

% Theorems and such

%%%%%
\theoremstyle{plain}\newtheorem{teo}{Theorem}[section]
\theoremstyle{plain}\newtheorem{prop}[teo]{Proposition}
\theoremstyle{plain}\newtheorem{lem}[teo]{Lemma}
\theoremstyle{plain}\newtheorem{cor}[teo]{Corollary}
\theoremstyle{definition}\newtheorem{defin}[teo]{Definition}
\theoremstyle{remark}\newtheorem{rem}[teo]{Remark}
\theoremstyle{definition}\newtheorem{example}[teo]{Example}

%\theoremstyle{plain}\newtheorem*{convention}{Convention}
%\theoremstyle{definition}\newtheorem*{defin0}{Definition}
%%%%%%%%%%%%%%%%%%%%

\begin{abstract}
We study focal points and Maslov index of a horizontal geodesic $\gamma:I\to M$ in the total
space of a semi-Riemannian submersion $\pi:M\to B$ by determining an explicit relation
with the corresponding objects along the projected geodesic $\pi\circ\gamma:I\to B$ in the base space.
We use this result to calculate the focal Maslov index  of a (spacelike) geodesic in a stationary
space-time which is orthogonal to a timelike  Killing vector field.
\end{abstract}

\maketitle
\tableofcontents

\begin{section}{Introduction}\label{sec:intro}
Riemannian submersions were introduced in the sixties by B. O'Neill and A. Gray (see \cite{Gra67,One66,One67})
as a tool to study the geometry of a Riemannian manifold with an additional structure in terms of certain
components, that is, the fibers and the base space.  A Riemannian submersion is a map $\pi:M\to B$
whose differential $\mathrm d\pi_p:T_pM\to T_{\pi(p)}B$ is surjective for all
 $p$, and such that its restriction to the space orthogonal to the fibers is an isometry with
$T_{\pi(p)}B$. The notion of Riemannian submersion can be naturally
extended to the case where the metric tensor is not positive  definite, i.e., to the semi-Riemannian case.
The novelty here is that both the total space and the base may have non positive definite metric tensors,
and the interesting observation is that the fibers of the submersion, which are embedded submanifolds
of the total space, are automatically non-degenerate.
Most of the results for Riemannian submersions, whose proofs involve only the relations between the Levi--Civita connections of the metrics of $M$ and of $B$ can be reproduced in the semi-Riemannian case by
adapting carefully the Riemannian proofs. For instance, a well known result by Hermann characterizes which Riemannian
fibrations are fiber bundles (see \cite{Herman60}); a similar result can be obtained in the semi-Riemannian
case with suitable modifications of Hermann's original proof (see Proposition \ref{hermanngener}). The situation becomes a little more involved when it gets
to conjugate or focal points, or to questions involving Morse index. Note in fact that conjugate/focal
points may accumulate in semi-Riemannian geometry (see \cite{PiTa03}), the Morse index is always infinite,
and its natural substitute, the \emph{Maslov index}, is not computed directly using the multiplicities
of the conjugate/focal points. In one of the classical papers by O'Neill (see \cite{One67}), the author
shows a correspondence between horizontal geodesics in the total space of a submersion and geodesics
in the base manifold, as well as relations between Jacobi fields, conjugate points and Morse index
of the two families of geodesics. This analysis cannot be carried over to the general semi-Riemannian
case with the techniques of \cite{One67}, and the goal of the present paper is to extend the results in
\cite{One67} to semi-Riemannian submersions. Our motivation for this kind of analysis comes
from specific examples of semi-Riemannian submersions, namely,  stationary and Kaluza-Klein spacetimes.

By a recent result of Javaloyes and S\'anchez (see \cite{JS08}) a distinguishing stationary spacetime
has a \emph{standard form}, i.e. given any complete timelike Killing vector field, there exists a
global spacetime decomposition of the form  $S\times \R$, with Killing field tangent to the
fiber $\R$ and $S$ a spacelike hypersurface. In this case, the projection onto $S$ is a semi-Riemannian
submersion, and the base space is Riemannian, thus one can prove a Riemannian global behavior for the
spacelike Lorentzian geodesics that are orthogonal to some complete timelike  Killing  vector field  (see Proposition \ref{thm:applstationary}). On the other hand,
in its most general setting,  the geometry of a Kaluza-Klein spacetime can be  described as
a  semi-Riemannian submersion $\pi\colon \tilde M\to M$, where $M$ is  the four dimensional spacetime,
while $\tilde M$ is the  multidimensional one.
We will show that conjugate points in a horizontal causal geodesic determine conjugate points in its projection  (see Proposition \ref{thm:applKaluzaKlein}).

The central issue in this paper is to study the Maslov index of a horizontal geodesic.
The Maslov index of a semi-Riemannian geodesic can be thought as an algebraic count of the singularities
of the exponential map along the geodesic. It was A. Helfer (see \cite{Hel94}) the first one to apply the abstract
notion of the Maslov index, that is, an intersection number in the Lagrangian Grassmannian of a symplectic space,
to the context of semi-Riemannian geodesics. This has been an essential tool in the development of Morse
theory and bifurcation theory for the strongly indefinite semi-Riemannian geodesic problem
(see for instance \cite{MerPicTau02, PiPoTa04}).
The main result of this paper is that the Maslov index of a horizontal geodesic $\gamma$ in the total
space of a semi-Riemannian submersion $\pi:M\to B$, and relative to an initial orthogonal manifold
$\mathcal Q=\pi^{-1}(\mathcal P)$, is equal to the Maslov index of the projected geodesic
$\pi\circ\gamma:I\to B$ relative to the initial orthogonal submanifold $\mathcal P$.
Under a certain (generic) non-degenerate situation, this is proved directly by showing
that corresponding conjugate points on $\gamma$ and on $\pi\circ\gamma$ give
the same contribution to the Maslov index (Proposition~\ref{contribution}).
For the general, i.e., possibly degenerate, case (Theorem~\ref{thm:eqMaslovindices}) the
proof is obtained by comparing the Lagrangian curves arising from the Jacobi equations
along the geodesics, and establishing a certain decomposition property of the Lagrangian
curve associated to the horizontal geodesic.

The Maslov index of a geodesic is computed using a trivialization of the tangent bundle
along the geodesic. In the standard literature, it is customary to define this index using
parallel trivializations along the geodesic.
However, in order to prove the equality between the Maslov indexes of a horizontal geodesic and
its projection, parallel trivializations do not work, because the projection onto the base
does not preserve parallelism. This forces us to use more general trivializations,
and some preliminary results on the independence of the Maslov index by arbitrary trivializations
are necessary (see Section~\ref{PreMaslov}).

A second problem to be studied is the question of lifting curves in the base to horizontal curves
in the total space, and lifting vector fields along curves in the base to infinitesimally horizontal
vector fields along horizontal curves. In Section~\ref{sec:horizontalgeo} we study this problem,
giving a geometric characterization for the \emph{derived vector field}, introduced by O'Neill in \cite{One67}.
A vector field along a horizontal curve is infinitesimally horizontal if it is the variational
vector field of a variation by other horizontal curves. The study of these fields goes through
an analysis of the tangent bundle to an abstract distribution $\mathcal D$ on a manifold $M$ endowed with a connection,
seen as a submanifold of the tangent bundle $TM$ (see Subsection~\ref{sub:horinfhor}). This
analysis uses the notion of second fundamental form of a distribution, which in the case
of the horizontal distribution of a semi-Riemannian submersion is computed in terms of the
fundamental tensors of the submersion (Subsection~\ref{sub:fundtensors}).
The relation between the second fundamental form of a semi-Riemannian  submanifold in the
base and the second fundamental form of its lift to the total space is studied in Proposition \ref{prop:secondfund}.

In the last part of the paper we study the index forms along a horizontal geodesic and its projection
(Theorem~\ref{indexrelation}), and we discuss a few applications of our results in stationary
and Kaluza-Klein spacetimes and
in bifurcation theory of geodesics (Subsection~\ref{sub:applications}).
\end{section}
\begin{section}{Preliminaries}\label{PreMaslov}
The literature on the notion of Maslov index and its applications is extremely rich, see for instance
references \cite{Gosson, Gosson2, Morvan, RobbinSalamon, SalamonZehnder}. We will give here a short account
of the basics needed for our purposes. The goal is proving an invariance property of the Maslov
index of a semi-Riemannian geodesic by arbitrary changes of trivialization of the tangent bundle 
along the geodesic (Proposition~\ref{thm:changeoftrivialization}).
\subsection{The Maslov index}
Let us consider a symplectic space $(V,\omega)$, with $\Dim(V)=2n$;
we will denote by $\Spl(V,\omega)$ the \emph{symplectic group} of $(V,\omega)$,
which is the closed Lie subgroup of $\GL(V)$ consisting of all isomorphisms
that preserve $\omega$.
A subspace $X\subset V$ is \emph{isotropic} if the restriction of $\omega$ to $X\times X$ vanishes
identically; an $n$-dimensional (i.e., maximal) isotropic subspace $L$ of $V$ is called
a \emph{Lagrangian subspace}. We denote by $\Lambda$ the Lagrangian Grassmannian of $(V,\omega)$,
which is the collection of all Lagrangian subspaces of $(V,\omega)$, and is a compact differentiable
manifold of dimension $\frac12n(n+1)$. A real-analytic atlas of charts on $\Lambda$ is given
as follows. Given a Lagrangian decomposition $(L_0,L_1)$ of $V$, i.e., $L_0,L_1\in\Lambda$ are transverse
Lagrangians, so that $V=L_0\oplus L_1$, then denote by $\Lambda^0(L_1)$ the open and dense subset
of $\Lambda$ consisting of all Lagrangians $L$ transverse to $L_1$. A diffeomorphism
$\varphi_{L_0,L_1}$ from $\Lambda^0(L_1)$ to the vector space $\Bsym(L_0)$ of all symmetric
bilinear forms on $L_0$ is defined by $\varphi_{L_0,L_1}(L)=\omega(T\cdot,\cdot)\vert_{L_0\times L_0}$,
where $T:L_0\to L_1$ is the unique linear map whose graph in $L_0\oplus L_1=V$ is $L$.
The kernel of $\varphi_{L_0,L_1}(L)$ is the space $L\cap L_0$.

Let us now briefly recall the notion of Maslov index for a continuous path $\ell:[a,b]\to\Lambda$.
More details on the theory can be found in \cite{PiTa08}.
For a fixed Lagrangian $L_0\in\Lambda$, the \emph{$L_0$-Maslov index }$\mu_{L_0}(\ell)$ of
$\ell$ is the half-integer characterized by the following properties:
\begin{itemize}
\item[(a)] $\mu_{L_0}$ is fixed-endpoint homotopy invariant;
\item[(b)] $\mu_{L_0}$ is additive by concatenation;
\item[(c)] if $\ell\big([a,b])\subset\Lambda^0(L_1)$ for some Lagrangian $L_1$ transverse
to $L_0$, then \[\mu_{L_0}(\ell)=\tfrac12\segnatura\big[\varphi_{L_0,L_1}\big(\ell(b)\big)\big]-
\tfrac12\segnatura\big[\varphi_{L_0,L_1}\big(\ell(a)\big)\big].\]
\end{itemize}
The Maslov index is invariant by symplectomorphisms, i.e., given two symplectic spaces
$(V_1,\omega_1)$, $(V_2,\omega_2)$, a continuous Lagrangian path $\ell:[a,b]\to\Lambda(V_1,\omega_1)$,
a fixed Lagrangian $L_0\in\Lambda(V_1,\omega_1)$ and a symplectomorphism $\phi_0:(V_1,\omega_1)\to(V_2,\omega_2)$,
then setting $\widetilde\ell(t)=\phi_0\big(\ell(t)\big)$ and $\widetilde L_0=\phi_0(L_0)$, one has
$\mu_{L_0}(\ell)=\mu_{\widetilde L_0}(\widetilde\ell)$.
Moreover, the Maslov index  is additive by direct sums of symplectic spaces in the following sense. Let $(V,\omega)=(V_1\oplus V_1,\omega_1\oplus\omega_2)$
be a decomposition of $(V,\omega)$ as the direct sum of the symplectic subspaces $(V_1,\omega_1)$ and $(V_2,\omega_2)$. Let $L_0$ and $\ell(t)$ be  Lagrangians of $(V,\omega)$ such that $L_0^i=L_0\cap V_i$ and $\ell^i(t)=\ell(t)\cap V_i$ are Lagrangians of $(V_i,\omega_i)$ for $i=1,2$ and $t\in[a,b]$. Then
\[\mu_{L_0}(\ell)=\mu_{L_0^1}(\ell^1)+\mu_{L_0^2}(\ell^2).\]
\begin{lem}\label{symplec}
Let $(V,\omega)$ and $(\tilde{V},\tilde{\omega})$ be two symplectic
spaces, $L_0$ a Lagrangian of $(V,\omega)$, $[a,b] \ni t\rightarrow {\ell}(t)$ a Lagrangian path in
$(V,\omega)$ and $\phi_t:V\rightarrow \tilde{V}$ with $t\in [a,b]$ a continuous family of symplectomorphisms
such that $\phi_t(L_0)$ is constantly equal to $\tilde{L}_0$. Consider the Lagrangian path in
$ (\tilde{V},\tilde{\omega})$ defined as $[a,b]\ni t\rightarrow\tilde{\ell}(t)=\phi_t(\ell (t))$.
Then the Maslov index $\mu_{L_0}(\ell)$ coincides with $\mu_{\tilde{L}_0}(\tilde{\ell})$.
\end{lem}
\begin{proof}
An easy homotopy argument, see \cite[Proposition~5.4.3, Proposition~5.4.5]{PiTa08}
\end{proof}
\subsection{Symplectic systems}
Consider the symplectic space $V=\R^n\oplus{\R^n}^*$ endowed with the symplectic form
$\omega\big((v,\alpha),(w,\beta)\big)=\beta(v)-\alpha(w)$. Let $\Spl(2n,\R)$ denote the symplectic
group, i.e., the group of all isomorphisms $T:V\to V$ that preserve $\omega$, and let
$\spl(2n,\R)$ denote its Lie algebra. Written in $n\times n$ blocks, an endomorphism
$\begin{pmatrix}A&B\cr C&D\end{pmatrix}$ belongs to $\spl(2n,\R)$ if and only if
$B$ and $C$ are symmetric and $D=-A^*$.

A \emph{symplectic system} is a system of differential equations:
\begin{equation}\label{eq:defsymplsystem}
\begin{pmatrix}v\cr\alpha\end{pmatrix}'=X(t)\begin{pmatrix}v\cr\alpha\end{pmatrix},
\end{equation}
where $X:[a,b]\to\spl(2n,\R)$ is a continuous map. The \emph{flow} of the symplectic
system \eqref{eq:defsymplsystem} is the $C^1$-curve $\Phi:[a,b]\to\Spl(2n,\R)$ satisfying
$\Phi'(t)=X(t)\Phi(t)$ and $\Phi(a)=\mathrm{Id}$.

Let $\Lambda$ denote the Lagrangian Grassmannian of $(V,\omega)$ and set $L_0=\{0\}\oplus{\R^n}^*$;
one has a curve $\ell:[a,b]\to\Lambda$ of class $C^1$ given by $\ell(t)=\Phi(t)\big[L_0\big]$.
We define the \emph{Maslov index} of the symplectic system \eqref{eq:defsymplsystem} to be the
$L_0$-Maslov index $\mu_{L_0}(\ell)$ of the curve $\ell$.

\smallskip

Consider two continuous maps $X_1,X_2:[a,b]\to\spl(2n,\R)$, with:
\[X_1=\begin{pmatrix}A_1&B_1\cr C_1&-A_1^*\end{pmatrix},\quad
X_2=\begin{pmatrix}A_2&B_2\cr C_2&-A_2^*\end{pmatrix};\]
the corresponding symplectic systems are \emph{isomorphic} if there exists a $C^2$-map
$Z:[a,b]\to\GL(\R^n)$ and a $C^1$-map $W:[a,b]\to\mathrm{Lin}(\R^n,{\R^n}^*)$ of symmetric
linear operators such that:
\begin{equation}\label{eq:ABC}
\begin{aligned}
&A_2=ZA_1Z^{-1}-ZB_1WZ^{-1}+Z'Z^{-1},\\
&B_2=ZB_1Z^*,\\
&C_2={Z^*}^{-1}(WA_1+C_1-WB_1W+A_1^*W+W')Z^{-1}.
\end{aligned}
\end{equation}
It is proved in \cite[Proposition~2.10.2]{london} that isomorphic symplectic systems have the same
Maslov index, under an (unnecessary) assumption that the final instant  is  non-conjugate.
The proof for the general case is obtained easily using the following facts:
\begin{itemize}
\item if $\ell_1,\ell_2:[a,b]\to\Lambda$ denote the curve of Lagrangians associated to the isomorphic
symplectic systems above,
then the following relation holds:
\begin{equation}\label{eq:elltildeell}
\ell_2(t)=\phi_0(t)\big[\ell_1(t)\big],\quad\forall\,t\in[a,b],
\end{equation}
where $\phi_0(t)$ is the symplectomorphism:
\begin{equation}\label{eq:phi0}
\phi_0(t)=\left(\begin{array}{cc}Z(t)&0\\ { Z(t)^*}^{-1}W(t)&{
Z(t)^*}^{-1}\end{array}\right).
\end{equation}
\item $\phi_0(t)$ preserves $L_0$ for all $t\in[a,b]$.
\item If $\phi_0$ is a continuous path of symplectomorphisms that preserve a Lagrangian $L_0$
and $\ell$ is any continuous curve in $\Lambda$, then the $L_0$-Maslov index of $\ell$ equals the
$L_0$-Maslov index of the curve $t\mapsto\phi_0(t)\big[\ell(t)\big]$ (see Lemma \ref{symplec}).
\end{itemize}
\end{section}
We observe that \eqref{eq:ABC} is equivalent to $X_2=\phi'_0\phi^{-1}_0+\phi_0 X_1\phi^{-1}_0$. Moreover, if $\Phi_1$ and $\Phi_2$ are the flows of the symplectic systems described by $X_1$ and $X_2$, then \eqref{eq:ABC} is also equivalent to $\Phi_2=\phi_0\Phi_1$.
\subsection{The symplectic system associated to a semi-Riemannian geodesic}\label{independence}
Let now $(M,g)$ be an $n$-dimensional semi-Riemannian manifold with Levi--Civita connection
$\nabla$, and let $\gamma:[a,b]\to M$ be a geodesic.
For all $t\in[a,b]$, let $\mathbf R_t:T_{\gamma(t)}M\to T_{\gamma(t)}M$ be the $g$-symmetric endomorphism
$R\big(\dot\gamma(t),\cdot\big)\dot\gamma(t)$, where $R$ is the curvature tensor of $\nabla$
chosen with the sign convention $R(X,Y)=[\nabla_X,\nabla_Y]-\nabla_{[X,Y]}$.
Consider any smooth ($C^2$) trivialization
\begin{equation}\label{eq:trivialization}
p(t):\R^n\stackrel\cong\longrightarrow T_{\gamma(t)}M
\end{equation}
of the pull-back bundle $\gamma^*TM$. Associated to this setup, we have the following objects:
\begin{itemize}
\item a continuous path $\widetilde g_t$ of non-degenerate symmetric bilinear forms on $\R^n$ defined
as the pull-back $\widetilde g_t=p(t)^*g_{\gamma(t)}$;
\item a continuous path of endomorphisms $\widetilde R(t):\R^n\to\R^n$ defined by the commutative diagram:
\[\xymatrix{T_{\gamma(t)}M\ar[r]^{\mathbf R_t}&T_{\gamma(t)}M\cr\R^n\ar[u]^{p(t)}\ar[r]_{\widetilde R(t)}&\R^n\ar[u]_{p(t)}}\]
For all $t$, $\widetilde R(t)$ is $\widetilde g_t$-symmetric.
\item A continuous map $\varpi:[a,b]\to\mathrm{gl}(\R^n)$ that relates the covariant derivative of
vector fields along $\gamma$ with the corresponding curves in $\R^n$ by the following formula:
\begin{equation}\label{covdev}
\Ddt\big[p(t)\widetilde v(t)\big]=p(t)\widetilde v'(t)+p(t)\varpi(t)\widetilde v(t),
\end{equation}
for all curve $\widetilde v:[a,b]\to\R^n$ of class $C^1$. Here $\Ddt$ is the covariant derivative along $\gamma$.
From an abstract viewpoint, $\varpi$ is the pull-back by $p:[a,b]\to\mathrm{FR}(TM)$ of the connection form
of the $\mathrm{GL}(n,\R)$-principal fiber bundle of all frames of $TM$.
\end{itemize}
The functions $\widetilde g$ and $\varpi$ are related by the following equality:
\begin{equation}\label{eq:relvarpitildeg}
\widetilde g'=\widetilde g\varpi+\varpi^*\widetilde g.
\end{equation}
Note that if the trivialization \eqref{eq:trivialization} is orthogonal, then $\widetilde g$ is constant,
and if  \eqref{eq:trivialization} is parallel, then $\varpi=0$.

Given a smooth curve $\widetilde v:[a,b]\to\R^n$, the corresponding vector field $v$ along $\gamma$ defined by
$v(t)=p(t)\widetilde v(t)$ is Jacobi if and only if $\widetilde v$ satisfies the second order linear equation:
\begin{equation}\label{eq:JacobiRn}
\ddt\big[\widetilde v'+\varpi\widetilde v\big]+\varpi\widetilde v'+\varpi^2\widetilde v+\widetilde R\widetilde v=0.
\end{equation}
Setting
\begin{equation}\label{alfa}
\alpha=\widetilde g\big[\widetilde v'+\varpi\widetilde v\big]:[a,b]\to{\R^n}^*
\end{equation}
Eq.~\eqref{eq:JacobiRn} can be rewritten as the system:
\[\begin{pmatrix}\widetilde v\cr\alpha\end{pmatrix}'=\begin{pmatrix}-\varpi&\widetilde g^{-1}\cr
-\widetilde g\widetilde R&\widetilde g'\widetilde g^{-1}-\widetilde g\varpi\widetilde g^{-1}\end{pmatrix}
\begin{pmatrix}\widetilde v\cr\alpha\end{pmatrix},\]
and using \eqref{eq:relvarpitildeg} this is the symplectic system:
\begin{equation}\label{eq:symplsystemassociatedgamma}
\begin{pmatrix}\widetilde v\cr\alpha\end{pmatrix}'=\begin{pmatrix}-\varpi&\widetilde g^{-1}\cr
-\widetilde g\widetilde R&\varpi^*\end{pmatrix}
\begin{pmatrix}\widetilde v\cr\alpha\end{pmatrix}.
\end{equation}
\begin{prop}\label{thm:changeoftrivialization}
The Maslov index of the symplectic system \eqref{eq:symplsystemassociatedgamma} does
\emph{not} depend on the choice of the trivialization $p$ in \eqref{eq:trivialization}.
\end{prop}
\begin{proof}
Assume that two trivializations $p(t):\R^n\to T_{\gamma(t)}M$ and $q(t):\R^n\to T_{\gamma(t)}M$ are given, and denote
by:
\begin{equation}\label{eq:duesistemi}
\begin{pmatrix}\widetilde v\cr\alpha\end{pmatrix}'=\begin{pmatrix}-\varpi&\widetilde g^{-1}\cr
-\widetilde g\widetilde R&\varpi^*\end{pmatrix}
\begin{pmatrix}\widetilde v\cr\alpha\end{pmatrix}, \quad
\begin{pmatrix}\widetilde w\cr\beta\end{pmatrix}'=\begin{pmatrix}-\tau&\widetilde h^{-1}\cr
-\widetilde h\widetilde S&\tau^*\end{pmatrix}
\begin{pmatrix}\widetilde w\cr\beta\end{pmatrix}
\end{equation}
the corresponding symplectic systems. Denote by $K(t):\R^n\to\R^n$ the isomorphism
$K(t)=q(t)^{-1}p(t)$; then, one computes easily:
\begin{align*}
&\widetilde g_t\!=\!K(t)^*\widetilde h_tK(t),\\
&\varpi(t)\!=\!K(t)^{-1}K'(t)+K(t)^{-1}\tau(t)K(t),\\
&\widetilde R(t)\!=\!K(t)^{-1}\widetilde S(t)K(t).
\end{align*}
for all $t$. Then it is easy to see that the symplectic systems in \eqref{eq:duesistemi} are isomorphic (with $Z=K^{-1}$ and $W\equiv0$),
and they have the same Maslov index.
\end{proof}
A totally analogous statement holds for the Maslov index of a geodesic relatively to an
initial orthogonal submanifold $\mathcal P$.

\begin{section}{Semi-Riemannian submersions and lifts}\label{lifts}
\subsection{Semi-Riemannian submersions}
Let $M$ and $B$  be differential manifolds having dimensions $n$ and $m$ respectively.
A submersion is a $C^\infty$ mapping $\pi:M\to B$
such that $\pi$ is of maximal rank. The implicit function theorem implies that
$\pi^{-1}(x)$ is a closed submanifold of $M$ for each $x\in B$, that we call a \emph{fiber}
of the submersion. When
 $M$ and $B$ are semi-Riemannian manifolds it is convenient to consider a special class
 of submersions, where one can relate in a good way the geometry of $M$ with the geometry
 of the fibers and of $B$. In particular, we  assume that the fibers  are  non-degenerate submanifolds, so that for
 every point $p\in M$ one has a decomposition of the tangent space $T_pM$ as an orthogonal direct sum:
\begin{equation}\label{eq:defcalVcalH}
T_pM=\V T_pM+\Hor T_pM.
\end{equation}
 Here $\V T_pM$ denotes the subspace of vectors tangent to the fiber through $p$  and
 $\Hor T_pM$ the orthogonal vectors to $\V T_pM$. These subspaces are respectively called
 \emph{vertical} and \emph{horizontal} subspaces, and we will denote by $\V$ and $\Hor$ the projections
 to the vertical and horizontal subspaces; $\V$ and $\Hor$ are smooth sections of the
vector bundle $\mathrm{Lin}(TM)$ of endomorphisms of $TM$.
\begin{defin}
Let $M$ and $B$ be semi-Riemannian manifolds, with metric tensors $g$ and $h$ respectively.
A semi-Riemannian submersion is a
submersion $\pi:M\rightarrow B$ such that\footnote{%
Observe that, in fact, (S2) implies (S1); namely, if $\mathrm d\pi$ is an isometry, then the horizontal
subspaces $\Hor T_pM$ are non-degenerate, and this implies that also the vertical subspaces, which are
their orthogonal, are non-degenerate.
However, keeping in mind the non-degeneracy of the fibers is important,
and it is useful to maintain (S1) in the definition of semi-Riemannian submersions.}
for every $p\in M$ (see \cite[p. 212]{One83})
\begin{enumerate}
\item[(S1)] the fiber $\pi^{-1}(x)$ is non degenerate ($ x=\pi(p)$)
\item[(S2)]  the differential map $\df\pi$
restricted to the horizontal subspace
\[\df\pi_p:\Hor T_pM\rightarrow T_xB\]
is an isometry.
\end{enumerate}
\end{defin}
There are very many situations where one has a natural semi-Riemannian submersion structure in a geometrical
problem; we will be interested in the following two examples.
\begin{example}[Stationary spacetimes]\label{exa:statspacetimesubm}
Let $(M,g)$ be a {\it Lorentzian manifold}, that is, a semi-Riemannian metric $g$ of index $1$.
We say that $M$ is {\it stationary} when there exists a timelike Killing field $Y$, i.e., a vector field $Y$,
with $g(Y,Y)<0$ and  ${\mathcal L}_Yg=0$, where ${\mathcal L}_Y$ is the Lie derivative with respect to $Y$.
The vector field $Y$ gives a timelike orientation to $(M,g)$, so that it becomes a spacetime.
We will consider the class of \emph{standard} stationary spacetimes, that is, those that can
be written as $(S\times\R,g)$ with
\[g\big((\xi,\tau),(\xi,\tau)\big)=g_0(\xi,\xi)+2\tau g_0(\delta,\xi)-\beta \tau^2,\]
where $\xi\in  TS$, $\tau\in \R$, $g_0$ is a Riemannian metric in $S$ and $\beta$ and $\delta$ are
respectively a positive function and a vector field in $S$. In this case, the timelike vector field
$Y$ is $\partial_t$ (where $t$ is the variable of $\R$).
If the stationary spacetime does not have too bad causal properties (more precisely, if it is
\emph{distinguishing}) and the timelike Killing field is complete, then it admits a standard splitting (see \cite{JS08}). We always can associate
a semi-Riemannian submersion to a  standard stationary spacetime $(S\times\R,g)$. Indeed,
the projection $\pi:(S\times\R,g)\rightarrow (S,\tilde{g})$ is a
semi-Riemannian submersion, where $\tilde{g}$ is the \emph{Riemannian}
metric defined by
$\tilde{g}(\xi,\xi)=g_0(\xi,\xi)+\frac{1}{\beta} g_0(\xi,\delta)^2$ for every $\xi\in TS$.
\end{example}

\begin{example}[Kaluza-Klein spacetimes]\label{exa:kk}
Let $\pi\colon M\to B$ be  a smooth principal $G$-bundle, $h$
a Lorentzian metric on $B$, $\bar g$ a $G$-invariant Riemannian metric on
the model space of the fibers, then by using a connection on  $M$, a
{\em Kaluza-Klein  metric}   $g$  on  $M$  is given by
\[g(X,X)= h(\df\pi(X),\df\pi(X))+\bar g(V,V),\]
where $V$ is the vertical component of $X$ (see \cite{Bourgu89}).
The projection $\pi\colon (M,g)\to (B,h)$ is  a Lorentzian submersion.
\end{example}
\begin{rem}
It is well known (see \cite{Herman60}) that if $(M,g)$  is a connected complete Riemannian manifold and $\pi\colon (M,g) \to (B,h)$ is a smooth surjective  Riemannian submersion, such that all the fibers are totally geodesic submanifolds, then all the fibers are isometric and $\pi$ is a smooth $G$-bundle with structure group $G$, the Lie group of the isometries of the fiber.
That result has been used to study ``Riemannian'' Kaluza-Klein theory taking a  submersion as starting point (see \cite[\S 4]{Hogan84} and also \cite[p.152-153]{Bourgu89}). It would be interesting to get a similar geometric characterization of a Kaluza-Klein spacetime (i.e. when
$\pi \colon (M,g)\to (B,h)$ is a Lorentzian submersion as in Example~\ref{exa:kk}).
Clearly if any curve in $B$  can be lifted to a globally defined  horizontal curve, then  $\pi$  is a fiber bundle and one can repeat the proof in \cite{Herman60} to get that $\pi$ is actually a smooth $G$-bundle.
\begin{prop}\label{hermanngener}
Let $(M,g)$ and $(B,h)$  be two connected semi-Riemannian manifolds and let $\pi\colon  (M,g) \to (B,h)$  be   a smooth  surjective  semi-Riemannian submersion  such that, for any $x\in B$, $\pi^{-1}(x)$ is totally geodesic.   Assume that either one of the following two assumptions is satisfied:
\begin{enumerate}
 \item $(M,g)$ is geodesically complete,
\item the fibers $\pi^{-1}(x)$ are compact and connected, for  each $x\in B$.
\end{enumerate}
Then $\pi$ is a smooth $G$-bundle.
\end{prop}
\begin{proof}
 As commented in the above remark, the only conditions necessary to repeat the Riemannian proof of Hermann \cite{Herman60}  in the semi-Riemannian context is the global definition of horizontal liftings. Assuming condition $(1)$ in the proposition, this can be shown by approximating a given curve  $\beta$  in $B$  by piecewise geodesics that, by the geodesic completeness, can be lifted to $M$. The sequence of lifts  converges,  up to consider a subsequence,   to the horizontal lift of $\beta$.  When assuming condition $(2)$, we can use a result by Ehresmann \cite[p.31]{Ehresm51} to show the existence of horizontal global lifts. To show that the fibers are isometric, for any vertical curve $\alpha$ we can consider a variation $(-\epsilon,\epsilon)\ni s\mapsto \alpha_s$ defined by a horizontal variational vector field. Analogously to the proof of \cite[Proposition 3.3]{Herman60} we can see that the energy of the curves $\alpha_s$ in the variation is constant and as a consequence the fibers are isometric.
\end{proof}
\end{rem}
\subsection{Fundamental tensors of a submersion}
\label{sub:fundtensors}
In \cite{One66}, B. O'Neill introduced the fundamental tensors $T$ and $A$ associated to
a semi-Riemannian submersion $\pi:M\to B$, defined as follows.
We denote  by $\nabla$ the Levi-Civita connection of $M$ and $\nabla^*$ the Levi-Civita connection
of $B$. Then for vectors fields
$E$ and $F$ in ${\mathfrak X}(M)$,
\begin{equation}\label{eq:deftensorT}
T_E(F)=\Hor \nabla_{\V E}(\V F)+\V \nabla_{\V E}(\Hor F).
\end{equation}
The other tensor can be thought as a dual tensor of $T$. In fact, it can be obtained
by reversing the role of $\V$ and $\Hor$. Explicitly,
\begin{equation}\label{eq:deftensorA}
A_E(F)=\Hor\nabla_{\Hor E}(\V F)+\V \nabla_{\Hor E}(\Hor F).
\end{equation}
The main properties of the fundamental tensors are the following:
\begin{enumerate}
\item
$T_E$ and $A_E$ are skew-symmetric operators on $T_pM$ for every $p\in M$,
\item $T$ is symmetric for vertical vector fields $V$ and $W$, i.e., $T_V(W)=T_W(V)$ and $A$ is alternating for horizontal vector fields $X$ and
$Y$, i.e., $A_X(Y)=-A_Y(X)$,
\item when restricted to vertical vector fields
$T$ coincides with the second fundamental form of the fibers, while $A$ coincides with
the integrability tensor of the horizontal distribution of $M$ when restricted to pairs of horizontal vectors.
\end{enumerate}
\subsection{The tangent space to a non-degenerate distribution}
Let us recall a few generalities on distributions. For $v\in TM$, consider the decomposition of $T_v(TM)=\mathrm{Hor}_v\oplus\mathrm{Ver}_v$
into the horizontal and vertical subspaces determined by the Levi--Civita connection of $g$.\footnote{Note that
this notion of horizontality and verticality of vectors in $TTM$ should not be confused with the notion of
horizontality of vectors in $TM$ associated to the distribution $\Hor T_pM$.} Assume that $\mathcal D\subset TM$ is any non-degenerate distribution
(i.e., $g\vert_{\mathcal D\times\mathcal D}$ is non-degenerate), and denote by
$\mathcal S^\mathcal D:TM\times\mathcal D\to\mathcal D^\perp$ its \emph{second fundamental form}, defined by:
\[\mathcal S_x^\mathcal D(v,w)=\mathbf p^\perp\big(\nabla_vW),\]
where $x\in M$, $v\in T_xM$, $w\in\mathcal D_x$, $W$ is any extension of $w$ to a local section of $\mathcal D$ and $\mathbf p^\perp$ is the
projection onto the orthogonal of $\mathcal D$.
Let now $x_0\in M$ and $v_0\in\mathcal D_{x_0}$ be fixed; we want to determine the tangent space
$T_{v_0}\mathcal D$, where $\mathcal D$ is considered merely as a submanifold of $TM$.
If $\left]-\varepsilon,\varepsilon\right[\ni t\mapsto v(t)\in\mathcal D$ is any smooth curve
with $v(0)=v_0\in\mathcal D_{x_0}$ and $x(t)=\pi^M\big(v(t)\big)$ is its projection onto $M$, then the horizontal and the vertical
component of the tangent vector $\dot v(0)\in T_{v_0}(TM)=\mathrm{Hor}_{v_0}\oplus\mathrm{Ver}_{v_0}$ are given respectively
by $\dot x(0)$ and $\Ddt v(0)$, where $\Ddt$ is the covariant derivative along the curve $x$.
Here, $\pi^M:TM\to M$ denotes the canonical projection.
By definition, \[\mathbf p^\perp\big(\Ddt v(0)\big)=\mathcal S_{x_0}^\mathcal D\big(\dot x(0),v_0\big).\]
\begin{lem}\label{thm:tangdistr}
For $v_0\in\mathcal D_{x_0}$, the tangent space $T_{v_0}\mathcal D$ is given by:
\begin{equation}\label{eq:tangentdistribution}
T_{v_0}\mathcal D=\big\{(u_1,u_2)\in\mathrm{Hor}_{v_0}\oplus\mathrm{Ver}_{v_0}:\mathbf p^\perp(u_2)=\mathcal S_{x_0}^\mathcal D(u_1,v_0)\big\}.
\end{equation}
\end{lem}
\begin{proof}
Denote by $C_{v_0}\subset T_{v_0}(TM)$ the right-hand side of \eqref{eq:tangentdistribution}.
We have shown above the inclusion $T_{v_0}\mathcal D\subset C_{v_0}$. Counting dimensions
one obtains immediately $T_{v_0}\mathcal D=C_{v_0}$.
\end{proof}

\subsection{Horizontal curves and infinitesimally horizontal vector fields}
\label{sub:horinfhor}
Let us now consider the case of a distribution $\mathcal D$ which is the horizontal distribution of
a semi-Riemannian submersion $\pi:M\to B$. In this case, the projection operator $\mathbf p^\perp$
coincides with the operator $\mathcal V$ defined in \eqref{eq:defcalVcalH}, and using \eqref{eq:deftensorT} and
\eqref{eq:deftensorA} one obtains easily the following equality:
\begin{equation}\label{eq:secffAT}
\mathcal S^\mathcal D(v,w)=A_v(w)+T_v(w)
\end{equation}
for all $v\in TM$ and all $w\in\mathcal D$.

Given a $C^1$ curve $\beta:[a,b]\to B$, a \emph{horizontal lift}
of $\beta$ is a curve $\gamma:[a,c]\subset [a,b]\to M$ such that $\pi\circ\gamma=\beta$
and $\dot\gamma(s)$ is horizontal for all $s\in [a,c]$.
\begin{prop}\label{thm:exlift}
Let $\pi:M\to B$ be a semi-Riemannian submersion and let $\mathcal D=\Ker(\mathrm d\pi)^\perp$ be the horizontal distribution of $\pi$. Given any $C^1$-curve $\beta:[a,b]\to B$ and any $p\in\pi^{-1}\big(\beta(a)\big)$,
there exists a unique maximal horizontal lift $\gamma:\left[a,c\right[\to M$ of $\beta$, with $c\le b$, satisfying
$\gamma(a)=p$. Such $\gamma$ has the same regularity as $\beta$.
\end{prop}
\begin{proof}
Using the local form of a submersion, for the local lifting problem it is not restrictive to assume
that $M=U\times V$ is the product of an open subset $U\subset\R^k$, $k=n-m$ and an open subset $V\subset\R^m$, such
that $B=V$ and $\pi:U\times V\to V$ is the projection onto the second factor. The horizontal distribution of the submersion is a distribution $\mathcal D$
on $U\times V$ such that, at each point $(x,y)\in U\times V$, $\mathcal D_{(x,y)}$ is a subspace
of $\R^n$ which is complementary to $\R^k\oplus\{0\}^m$, hence it is the graph of a linear
map $F(x,y):\R^m\to\R^k$.  The map $U\times V\ni(x,y)\mapsto F(x,y)\in\mathrm{Lin}(\R^m,\R^k)$ is
smooth. Given a curve $x:[a,b]\to V$, a lifting of $x$ is a curve $(x,y):[a,b]\to U\times V$
where $y$ satisfies the following ODE:
\[y'(t)=F\big(x(t),y(t)\big) x'(t).\]
The thesis follows easily using standard existence and uniqueness results for ODE's in $\R^n$.
\end{proof}
Let us study the problem of lifting horizontally a $1$-parameter family of curves
by considering the following situation.

Let us assume we are given a $C^1$-map $\Psi:[a,b]\times\left]-\varepsilon,\varepsilon\right[\to B$
and a $C^1$-curve $\eta:\left]-\varepsilon,\varepsilon\right[\to M$ such that $\pi\circ\eta=\Psi(a,\cdot)$.
Proposition~\ref{thm:exlift} gives us the existence of a map $\Gamma:\mathcal A\to M$ defined on an open
subset $\mathcal A$ of $\R^2$ that contains the segment $\{a \}\times\left]-\varepsilon,\varepsilon\right[$
such that $\pi\big(\Gamma(t,s)\big)=\Psi(t,s)$ for all $(t,s)\in\mathcal A$, and with the property that
the curve $t\mapsto\Gamma(t,s)$ is of class $C^1$ for all $s\in\left]-\varepsilon,\varepsilon\right[$.
Such map $\Gamma$ is defined by the property that for every $s\in\left]-\varepsilon,\varepsilon\right[$,
the map $t\mapsto\Gamma(t,s)$ is a maximal horizontal lift of the curve $t\mapsto\Psi(t,s)$ satisfying
$\Gamma(a,s)=\eta(s)$. Let us denote by $\gamma$ the curve $t\mapsto\Gamma(t,0)$.
\begin{prop}\label{thm:varhoriz}
With the notations above, we have:
\begin{itemize}
\item[(a)] the map $\Gamma$ is $C^1$ in $\mathcal A$.
\end{itemize}
If $\mathcal A$ contains the segment $[a,b]\times\{0\}$, then:
\begin{itemize}
\item[(b)] $\mathcal A$ contains the rectangle $[a,b]\times[-\delta,\delta]$ for some $\delta\in\left]0,\varepsilon\right[$;
\item[(c)] the variational vector field $E=\frac{\partial}{\partial s}\big\vert_{s=0}\Gamma(t,s)$ along $\gamma$ satisfies
the identity:
\begin{equation}\label{eq:infhorizontality}
\V\Ddt E=\mathcal S_{\gamma(t)}^\mathcal D\big(E(t),\dot\gamma(t)\big)=A_E(\dot\gamma)+T_{E}(\dot\gamma)=
A_{\Hor E}(\dot\gamma)+T_{\V E}(\dot\gamma).
\end{equation}
\end{itemize}
\end{prop}
\begin{proof}
Part (a) follows from standard smooth dependence results on the data for solutions of ODE's.
Also part (b) is obtained easily from standard continuity results for ODE's.
For part (c), observe that $\frac\partial{\partial t}\Gamma(t,s)\in\mathcal D_{\Gamma(t,s)}$,
and thus
\begin{multline*}\frac{\partial}{\partial s}\Big\vert_{s=0}
\frac\partial{\partial t}\Gamma(t,s)=
\Big(\frac{\partial}{\partial s}\Big\vert_{s=0}\Gamma,\frac{\mathrm D}{\mathrm ds}\Big\vert_{s=0}\frac\partial{\partial t}\Gamma(t,s)\Big)
\\=\Big(\frac{\partial}{\partial s}\Big\vert_{s=0}\Gamma,\frac{\mathrm D}{\mathrm dt}\frac{\partial}{\partial s}
\Big\vert_{s=0}\Gamma(t,s)\Big)=\big(E,\Ddt E\big)\in T_{\dot\gamma(t)}\mathcal D.
\end{multline*}
The conclusion follows now from  Lemma~\ref{thm:tangdistr} and formula \eqref{eq:secffAT}.
\end{proof}
\subsection{Basic horizontal vector fields}
A horizontal vector field is \emph{basic} when it is $\pi$-related
with a vector field on  $B$.
The following result can be found in \cite[Lemmas 1 and 3]{One66}.

\begin{lem}\label{nablas}
Let $X$ and $Y$ be horizontal vector fields and $V$ a vertical vector field. Then the following identities hold:
\begin{enumerate}
\item $\nabla_VX=\Hor \nabla_VX+T_V(X)$ (if $X$ basic, $\Hor\nabla_VX=A_X(V)$).
\item $\nabla_XV=A_X(V)+\V \nabla_XV$.
\item $\nabla_XY=\Hor \nabla_XY+A_X(Y)$.
\item $\Hor\nabla_XY$ is basic and is $\pi$-related to $\nabla^*_{\df \pi(X)}(\df\pi(Y))$.\qed
\end{enumerate}
\end{lem}

Given a submanifold ${\mathcal P}\subset B$, the total lift ${\mathcal Q}=\pi^{-1}({\mathcal P})$ of
${\mathcal P}$ by the submersion $\pi$
is a submanifold of $M$, as it can be proven using the inverse mapping theorem.
If the submersion is semi-Riemannian, the total lift
of a non degenerate submanifold is also non degenerate.
 Namely, for $p\in \mathcal Q$, the tangent space $T_p\mathcal Q$ is the sum of the orthogonal
subspaces $\mathcal VT_pM$ and $\mathrm d\pi_p^{-1}\big(T_{\pi_p}\mathcal P\big)\cap\mathcal HT_pM$.
They are both non-degenerate subspaces of $T_pM$, thus $T_p\mathcal Q$ is non-degenerate.

\begin{rem}
 The differential $\df \pi$ gives an identification of horizontal vectors in $M$ with tangent
 vectors to the space $B$. Such identification will be use implicitly throughout the paper.
 We observe that with this identification, horizontal tangent vectors to $\mathcal Q$ correspond
 to tangent vectors to $\mathcal P$.
\end{rem}

We want to show that there is a relation between
the second fundamental form of ${\mathcal P}$ and the second fundamental form of its total lift.
In fact, we will consider the tensor
${\mathcal S}^{\mathcal P}:T{\mathcal P}\times T{\mathcal P}^\bot\to T{\mathcal P}$ defined as
\[{\mathcal S}^{\mathcal P}_x(v,w)=(\nabla_vW)^{\mathrm t},\]
where $x\in {\mathcal P}$, $v\in T_x{\mathcal P}$, $w\in (T_x{\mathcal P})^\bot$, $W$ is any extension of $w$ to a orthogonal vector field to $\mathcal P$ and the superscript  $\mathrm t$ denotes the tangent part
to ${\mathcal P}$.
   Analogously, we define ${\mathcal S}^{\mathcal Q}:T{\mathcal Q}\times T{\mathcal Q}^\bot\to T{\mathcal Q}$.
\begin{prop}\label{prop:secondfund}
If $V$ is a vertical vector and $X$ and $Z$ are in $\Hor (T{\mathcal Q})$, then
\begin{align}\label{ver}
{\mathcal S}^{\mathcal Q}(V,Z)&=T_V(Z)+A_Z(V)^{\mathrm t},\\
\label{hor}
{\mathcal S}^{\mathcal Q}(X,Z)&=A_X(Z)+{\mathcal S}^{\mathcal P}(X,Z),
\end{align}
where $\mathrm t$ denotes the tangent part to $\mathcal P$ for vectors in $TB$ and the tangent part to $\mathcal Q$ for vectors in $TM$.
\end{prop}
\begin{proof}
By part $(1)$ in Lemma \ref{nablas} we see that
$${\mathcal S}^{\mathcal Q}(V,Z)=(\Hor\nabla_VZ)^{\mathrm t}+T_V(Z).$$
From this expression we deduce that $(\Hor\nabla_VZ)^{\mathrm t}$
is tensorial in $V$ and $Z$, so we can assume that $Z$ is
basic, and using again part $(1)$ in Lemma \ref{nablas} we obtain $(\Hor\nabla_VZ)^{\mathrm t}=A_Z(V)^{\mathrm t}$,
which concludes the proof of \eqref{ver}. Equation \eqref{hor} follows
directly from $(3)$ and $(4)$ in Lemma \ref{nablas}.
\end{proof}
\end{section}
\begin{section}{Horizontal geodesics}\label{sec:horizontalgeo}
Our aim in this section is to reformulate some of the main results in \cite{One67}, with some minor
modifications that clarify the role of the derived vector field introduced in Definition~\ref{thm:defderivedvectorfield}.
\subsection{The derived vector field}
Given a vector field $E$ along a curve $\alpha$ in $M$, we will denote
$E_*$ its projection by $\df\pi$, which is a vector field along the curve $\pi\circ\alpha$.
Moreover, we will use the same notation for a vector in $\pi\circ \alpha$ and its
horizontal lift in $\alpha$. It is important to clarify,
as we will use the same notation for the covariant derivative in $M$ and $B$,
that $\Ddt E_*$, when identified with a vector field on $\alpha$, denotes the horizontal lift of
the covariant derivative in $B$ of $E_*$. Even if the results of
\cite{One67} are stated and proved only in the case of Riemannian submersions, most of them are still valid
in the semi-Riemannian context. We will shortly recall in this section a few basic facts from \cite{One67},
stated for semi-Riemannian submersions.

It will be useful to introduce the following notation. Given a smooth  curve $\alpha$ in $M$
and a smooth vector field $E$ along $\alpha$, write $E=H+V$, with $H=\Hor E$ and $V=\V E$.
\begin{teo}\label{covariant}
Let $\pi:M\rightarrow B$ be a semi-Riemannian submersion, and let $E=H+V$
be a vector field on a curve $\alpha$ in $M$. Then
\begin{align*}
\Hor(\Ddt E)&=\Ddt E_*+A_H(\V\dot\alpha)+A_{\Hor \dot\alpha}(V)+T_{\V \dot\alpha}(V),\\
\V(\Ddt E)&=A_{\Hor\dot\alpha}(H)+T_{\V\dot\alpha}(H)+\V(\Ddt V).
\end{align*}
\end{teo}
\begin{proof}
See \cite[Theorem 1]{One67}.
\end{proof}
\begin{cor}
Let $\alpha$ be a curve in $M$ with $X=\Hor \dot\alpha$ and $U=\V\dot\alpha$. Then
\begin{align*}
\Hor(\Ddt\dot\alpha)&=\Ddt\dot\alpha_*+2A_X(U)+T_U(U),\\
\V(\Ddt\dot\alpha)&=T_U(X)+\V(\Ddt U),
\end{align*}
where $\Ddt\dot\alpha_*$ is the horizontal lift to $\alpha$ of the acceleration of $\pi\circ\alpha$ in $B$.
\end{cor}
\begin{cor}
The  projection of a horizontal geodesic in $M$ is a geodesic in $B$.\qed
\end{cor}
\begin{teo}\label{geohor}
Let $\pi:M\rightarrow B$ be a semi-Riemannian submersion. If $\gamma$ is a geodesic of $M$
that is horizontal at some point, then it is always horizontal (hence $\pi\circ\gamma$
is a geodesic of $B$).
\end{teo}
\begin{proof}
See \cite[Corollary 2]{One67}.
\end{proof}

\begin{defin}\label{thm:defderivedvectorfield}
The \emph{derived} vector field of $E=H+V$, denoted by $D(E)$, is the smooth vector field
along $\gamma$ defined by:
\[D(E)=\V\left(\Ddt V\right)-T_V(\dot\gamma)+2A_{\dot\gamma}(H).\]
\end{defin}
\begin{lem}\label{lem:D=0}
Let $\gamma$ be a horizontal curve in $M$. Given a vector field $P$ on
$\pi\circ\gamma$ and a vector $z\in T_{\gamma(t_0)}M$, there exists a unique vector
field $E$ on $\gamma$ such that
\begin{enumerate}
\item $E_*=P$,
\item $D(E)=0$,
\item $E(t_0)=z$.
\end{enumerate}
\end{lem}
\begin{proof}Analogous to Lemma 1 in \cite{One67}.
\end{proof}
\subsection{Infinitesimally horizontal vector fields}
An immediate calculation using \eqref{eq:deftensorT} and \eqref{eq:deftensorA} shows that
equality \eqref{eq:infhorizontality} is satisfied if and only if $D(E)=0$.
Thus, we have the following:
\begin{prop}\label{d=0}
Let $\gamma$ be a smooth horizontal curve in $M$ and let $E$ be a smooth vector field.
Then, $E$ is \emph{infinitesimally horizontal}, i.e., it is the variational vector field
along $\gamma$ corresponding to a smooth variation of $\gamma$ by horizontal curves, if and only if $D(E)=0$.
\end{prop}
\begin{proof}
The proof of Proposition~\ref{thm:varhoriz} and the second equation in Theorem~\ref{covariant}  show
that equation \eqref{eq:infhorizontality},
equivalent to $D(E)=0$,
is precisely the linearization of the horizontality condition for variations of $\gamma$.
Thus, if $s\mapsto\gamma_s$ is a smooth variation of $\gamma$ by horizontal curves, then
necessarily the variational vector field $E=\frac{\partial}{\partial s}\big\vert_{s=0}\gamma_s$
satisfies $D(E)=0$. Conversely, assume that $D(E)=0$, and consider the projection $P=E_*$.  We can choose
a variation of $\pi\circ\gamma$ in $B$ with variational vector field $P$, and by Proposition~\ref{thm:varhoriz}
we can lift such variation to a variation of $\gamma$ in $M$  by horizontal curves. Moreover,
we can prescribe the initial value of the corresponding variational field $\widetilde P$ to be
equal to the value of $E$. Thus, by the first part of the proof,  $\tilde P$ is an infinitesimally horizontal
 vector field that projects onto $P$,
and it coincides with $E$ at the initial instant, therefore, by Lemma \ref{lem:D=0}, $\widetilde P=E$ everywhere.
\end{proof}

We recall briefly the notion of Jacobi fields. Given a geodesic $\gamma$ in $M$, a field $E$ on $\gamma$ is Jacobi iff satisfies the Jacobi equation
$\tfrac{\mathrm D^2}{\mathrm dt^2} E=R(\dot\gamma,E)\dot\gamma$, where $R$ is the curvature tensor of $M$.
Analogously, given a geodesic $x$ in $B$, a field $F$ on $x$ is Jacobi iff satisfies
$\tfrac{\mathrm D^2}{\mathrm dt^2}F= R_*(\dot x,F)\dot x$,  where $R_*$ is the curvature tensor of $B$.
\begin{teo}
Let $E$ be a vector field on a horizontal geodesic $\gamma$ in $M$. Then
\begin{align*}
\Hor \big(\tfrac{\mathrm D^2}{\mathrm dt^2}E- R(\dot\gamma,E)\dot\gamma\big)&=\tfrac{\mathrm D^2}{\mathrm dt^2} E_*- R_*(\dot\gamma,E_*)\dot\gamma+2A_{\dot\gamma}(D),\\
\V\big(\tfrac{\mathrm D^2}{\mathrm dt^2}E- R(\dot\gamma,E)\dot\gamma\big)&=\V(\Ddt D)+T_D(\dot\gamma),
\end{align*}
where $D=D(E)$ is the derived vector field of $E$ and  $R$
and $ R_*$ are the curvature tensors of $M$ and $B$,  and $\tfrac{\mathrm D^2}{\mathrm dt^2} E_*- R_*(\dot\gamma,E_*)\dot\gamma$ denotes the horizontal lift to $\gamma$ of the vector field $\tfrac{\mathrm D^2}{\mathrm dt^2} E_*-R_*(\dot x,E_*)\dot x$, along the curve $x=\pi\circ\gamma$.
\end{teo}
\begin{proof}
See \cite[Theorem 2]{One67}.
\end{proof}
\begin{cor}\label{relacionajacobi}
A field $E$ on a horizontal geodesic in $M$ with derived vector field $D(E)=0$
is Jacobi  if and only if $P=E_*$ is a Jacobi field of $\pi\circ\gamma$ in $B$.\qed
\end{cor}
\end{section}
\begin{section}{Maslov index of horizontal geodesics}
We will now relate the Maslov index of a horizontal geodesic in a semi-Rie\-mannian
submersion with the Maslov index of its projection on the base of the submersion.
We emphasize that, unlike the standard Riemannian case, establishing a correspondence
between Jacobi fields on the horizontal geodesic and on its projection is not sufficient
in order to prove equality of the Maslov indexes. Namely, the Maslov index does not
depend only on the dimension of the space of Jacobi fields that are zero in the origin
and in the conjugate instant.
\smallskip

Throughout this section, we will consider the following setup:
\begin{itemize}
\item $\pi: (M,g)\to (B,h)$ is a semi-Riemannian submersion;
\item $\gamma:[a,b]\to M$ is a horizontal geodesic;
\item $x:[a,b]\to B$ is the projected geodesic in $B$: $x=\pi\circ\gamma$;
\item $\mathcal P\subset B$ is a non-degenerate submanifold of $B$ with $x(a)\in\mathcal P$ and $\dot x(a)\in T_{x(a)}\mathcal P^\perp$;
\item  $\mathcal S^\mathcal P$ is the second fundamental form of $\mathcal P$ at the point $x(a)$ in the direction $\dot x(a)$;
\item $\mathcal Q=\pi^{-1}(\mathcal P)$;
\item  $\mathcal S^\mathcal Q$ is the second fundamental form of $\mathcal Q$ at the point $\gamma(a)$ in the direction $\dot\gamma(a)$.
\end{itemize}
\smallskip
\subsection{$\mathcal Q$-Jacobi fields}
Consider the initial (or final) orthogonal submanifold $\mathcal Q$ for the geodesic $\gamma$.
One has the notion of a \emph{$\mathcal{Q}$-Jacobi field}, that is, a Jacobi field $E$ along $\gamma$
such that $E(a)\in T_{\gamma(a)}\mathcal{Q}$ and
\begin{equation}\label{pjacobi}
\left(\Ddt E(a)\right)^{\mathrm t}={\mathcal S}^{\mathcal Q}(E(a),\dot\gamma(a)).
\end{equation}
The $\mathcal{Q}$-Jacobi fields are precisely the variational vector fields corresponding to
variations of $\gamma$ given by geodesics that start orthogonally to $\mathcal{Q}$  (see \cite[Chapter 10]{One83}).
When a non-null $\mathcal{Q}$-Jacobi field is zero at an instant $t_0>a$, we say that
$\gamma(t_0)$ is a \emph{$\mathcal Q$-focal point} or that $t_0$ is a {\em $\mathcal Q$-focal instant} of the geodesic $\gamma$.  The idea behind the name is that there exists a
continuum of geodesics departing orthogonally from the submanifold $\mathcal{Q}$ and
focusing at $\gamma(t_0)$, but this holds  only
up to first order infinitesimals.

We will use several families of Jacobi fields, and it will be useful to introduce the following notation.
Let $\J$ be the linear space of all Jacobi fields along $\gamma$ and $\J_*$ be the linear space of Jacobi fields along $x$, then we define
\[\J_*(\mathcal{P})=\big\{E\in \J_* :  \text{$E$ is a ${\mathcal{P}}$-Jacobi field}\big\}\]
and
\[\J_\delta({\mathcal{P}})=\big\{E\in\J:D(E)=0,\,E_*\in \J_*(\mathcal{P})\big\}.\]
In particular, $\J_\delta(x(a))$ are the Jacobi fields   $J$ along $\gamma$ that are zero in $\gamma(a)$ and $D(J)=0$,  and $\J_*(x(a))$, the Jacobi fields along $x$ that are zero in $x(a)$.
\begin{lem}\label{lem:equiv}
A vector field $E$ on $\gamma$, such that $E(a)\in T_{\gamma(a)}\mathcal Q$, satisfies
\begin{equation}\label{boundaryconditions}
\left(\Ddt E(a)\right)^{\mathrm t}-{\mathcal S}^{\mathcal Q}(E(a),\dot\gamma(a))=\left(\Ddt E_*(a)\right)^{\mathrm t}-{\mathcal S}^{\mathcal P}(H(a),\dot\gamma(a))+ D(E)(a),
\end{equation}
where $H=\Hor E$. Moreover, $E$
 is a
$\mathcal{Q}$-Jacobi field iff
$E\in \J_\delta(\mathcal{P})$.
\end{lem}
\begin{proof}
 From \eqref{ver} and \eqref{hor}, we get
\begin{equation}\label{segundaforma}
{\mathcal S}^{\mathcal Q}(E(a),\dot\gamma(a))=T_{V}(\dot\gamma)(a)+ A_{\dot\gamma}(V)^{\mathrm t}(a)+A_H (\dot\gamma)(a)+{\mathcal S}^{\mathcal P}(H(a),\dot\gamma(a)),
\end{equation}
where $V=\V E$,  and from Theorem \ref{covariant} and Definition \ref{thm:defderivedvectorfield}, we deduce the equation
\begin{equation}\label{edea}
\Ddt E=\Ddt E_*+A_{\dot\gamma}(V)-A_{\dot\gamma}(H)+T_V(\dot\gamma)+D(E).
\end{equation}
Equation \eqref{boundaryconditions} follows from \eqref{segundaforma} and \eqref{edea} by taking into account that $A$ is alternating for horizontal vectors. For the second part, we observe that \eqref{boundaryconditions} implies  that
condition \eqref{pjacobi} for $E$ to be a ${\mathcal Q}$-Jacobi field is equivalent to
\begin{equation}\label{ultimate}\left(\Ddt E_*(a)\right)^{\mathrm t}={\mathcal S}^{\mathcal P}(H(a),\dot\gamma(a))
-D(E)(a).
\end{equation}
If $E$ is a ${\mathcal Q}$-Jacobi field, then it is the variational vector field of a variation of $\gamma$ by geodesics that are horizontal at least at the origin, but by Theorem \ref{geohor} they have to be horizontal wherever. Applying Proposition \ref{d=0} we obtain that $D(E)=0$, so that
 Eq. \eqref{ultimate} and Corollary \ref{relacionajacobi}  imply that $E_*$ is $\mathcal{P}$-Jacobi.  For the converse, observe that using    \eqref{boundaryconditions} and Corollary \ref{relacionajacobi} we deduce that every $E$ in $\J_\delta({\mathcal P})$ is a $\mathcal{Q}$-Jacobi field.
\end{proof}
\subsection{$\mathcal Q$-Maslov index}
The  $\mathcal{Q}$-Maslov index of a geodesic is an algebraic count of the
$\mathcal{Q}$-focal points obtained as follows. Consider a smooth trivialization \begin{equation*}%\label{eq:trivialization}
p(t):\R^n\stackrel\cong\longrightarrow T_{\gamma(t)}M
\end{equation*}
of the tangent bundle $TM$ along the geodesic $\gamma$. Define
\begin{equation}\label{LQ}
 L_{\mathcal Q}=\{(v,w)\in \R^n\times{\R^n}^*: (p(a) [w] )^{\mathrm t}={\mathcal S}^{\mathcal Q}(p(a)[v] ,\dot\gamma(a))\},
\end{equation}
and let $[a,b]\ni t\to\Phi(t)$ be the flow of the symplectic system \eqref{eq:symplsystemassociatedgamma}  considering the symplectic space $V=\R^n\oplus{\R^n}^*$ endowed with the symplectic form
$\omega\big((v,\alpha),(w,\beta)\big)=\beta(v)-\alpha(w)$.
Then, the $\mathcal Q$-Maslov index of $\gamma$ is the Maslov index of the Lagrangian path $t\to  \Phi(t)[L_{\mathcal Q}]$  with respect to the Lagrangian $L_0=\{0\}\times {\R^n}^*$. As in Section \ref{independence} it is possible to show that the
$\mathcal Q$-Maslov index does not depend on the trivialization $p$.

As a first step, we will show that there is a correspondence between the non-degenerate $\mathcal{Q}$-focal and
$\mathcal{P}$-focal points of $\gamma$ an $x$ and the contribution to the Maslov index.
\begin{lem}\label{lem:suma}
Consider the following subspaces of $T_{\gamma(t_0)}M$ and $T_{x(t_0)}B$ respectively,
$\J_\delta(\mathcal{P})[t_0]=\big\{J(t_0): J\in J_\delta(\mathcal{P})\big\}$ and
$\J_*(\mathcal{P})[t_0]=\big\{J(t_0): J\in {\mathcal J}_*(\mathcal{P})\big\}$
for every $t_0\in\left(a,b\right]$. Then
\begin{equation}\label{adition}
\J_\delta(\mathcal{P})[t_0]=\V T_{\gamma(t_0)}M+\J_*(\mathcal{P})[t_0].
\end{equation}
\end{lem}
\begin{proof}
 Fix $J_*\in\J_*(\mathcal{P})$, then by Lemma \ref{lem:D=0} and
Corollary \ref{relacionajacobi}, there exists a Jacobi field $ J_v$ in $\J_\delta(\mathcal{P})$
such that $J_v(t_0)=J_*(t_0)+v$ for every $v\in \mathcal{V}T_{\gamma(t_0)}M$. Conversely,
if $J\in\J_\delta(\mathcal{P})$, then $J_*\in \J_*(\mathcal{P})$ and $J(t_0)=J_*(t_0)+v$
for some $v\in \mathcal{V}T_{\gamma(t_0)}M$.
\end{proof}

\begin{prop}\label{contribution}
Let $\pi:M\rightarrow B$ be a semi-Riemannian submersion and $\gamma:[a,b]\rightarrow M$  a
horizontal geodesic. Then, an instant $t_0$ is a $\mathcal{Q}$-focal instant
of $\gamma$  if and only if it is a $\mathcal{P}$-focal instant
of the curve $x=\pi\circ\gamma$.
Furthermore, if $t_0$ is a non-degenerate focal instant of $\gamma$, then so it is of $x$
and the contribution to the Maslov index given by $\gamma(t_0)$ coincides with the one of
$x(t_0)$.
\end{prop}
\begin{proof}
We recall that  $t_0$ is a $\mathcal P$-focal instant of $x$ when $\J_*(\mathcal{P})[t_0]^\bot$ is
not equal to $\{0\}$ and that in such a case, its contribution to the Maslov index coincides with the signature of $h$
restricted to the subspace $\J_*(\mathcal{P})[t_0]^\bot$  (see for instance \cite{MerPicTau02}).
Moreover, by Lemma \ref{lem:equiv}, the $\mathcal Q$-focal points of $\gamma$ are just the instants where $\J_\delta(\mathcal{P})[t_0]^\bot$ is non trivial, and its contribution, the signature
of $g$ restricted to such a space.
Lemma \ref{lem:suma} implies that $\J_*(\mathcal{P})[t_0]^\bot$ is isometric to $\J_\delta(\mathcal{P})[t_0]^\bot$, so that the thesis of the proposition follows.
\end{proof}

We proceed now to prove the equality between the Maslov indexes in the general case.
\begin{teo}\label{thm:eqMaslovindices}
Let $\pi:(M,g)\rightarrow (B,h)$ be a semi-Riemannian submersion,  $\gamma:[a,b]\rightarrow M$ a
horizontal geodesic and $\mathcal{P}$ a submanifold of $B$ through $x (a)$ orthogonal to $x$. Then, the $\mathcal{Q}$-Maslov index of
$\gamma$  coincides with the ${\mathcal P}$-Maslov index
of the geodesic $x$.
\end{teo}
\begin{proof}
 Consider an orthonormal frame
 \[p(t)=\{E^1(t),\dots,E^m(t), F^1(t),\dots,F^{n-m}(t)\}\] along $\gamma$ such that
 $E^1(t),\dots,E^m(t)$ are horizontal   and  $F^1(t),\dots,F^{n-m}(t)$ are vertical for
 every $t\in[a,b]$. Clearly, $p_*(t)=\{E^1_*(t),\dots,E^m_*(t)\}$ is an orthonormal
 frame along $x$.  Define $\mathcal{V}=\{0\}\times\R^{n-m}\subset \R^n$,
$\mathcal{H}=\R^m\times\{0\}\subset \R^n$. We recall that
$J$ is $\mathcal{Q}$-Jacobi iff $D(J)=0$ and $J_*\in\J_*(\mathcal{P})$ (see Lemma \ref{lem:equiv}).
Now define the Lagrangian path
\begin{align*}
\ell(t)=\big\{\big(p(t)^{-1}(J(t)), p(t)^*g_{\gamma(t)}[\Ddt J(t)]\big)\in \R^n\times {\R^n}^*:J\in\J_\delta(\mathcal{P})\big\}.
\end{align*}
for $t\in[a,b]$.  The $\mathcal Q$-Maslov index of $\gamma$ is equal to $\mu_{L_0}(\ell)$ (see \eqref{covdev}, \eqref{eq:JacobiRn}, \eqref{alfa}, \eqref{LQ} and Lemma \ref{lem:equiv}).   In the following, we will make an abuse of notation omitting $p(t)$ and $p^*(t)$   to avoid cluster.  Using that $D(J)=0$, \eqref{edea} and Lemma \ref{lem:suma} we deduce that
\begin{multline*}
 \ell(t)=\big\{(J_*(t)+V,g_{\gamma(t)}\left[\Ddt J_*(t)-A_{\dot\gamma}(J_*(t))\right.\\\left.+A_{\dot\gamma}(V)+T_V(\dot\gamma) \right])\in \R^n\times {\R^n}^*:J_*\in\J_*(\mathcal{P});V\in\mathcal{V}\big\}.
\end{multline*}
On the other hand,  let us define
\[\ell_*(t)=\big\{(J_*(t),h_{x(t)}\left[\Ddt J_*(t)\right])\in{\R^m}\times{\R^m}^*:J_*\in\J_*(\mathcal{P})\big\},\]
 (here we also omit $p_*(t)$ and $p_*(t)^*$).  The $\mathcal{P}$-Maslov index of $x$ is equal to the Maslov index
of the Lagrangian path $t\to \ell_*(t)$  relatively to the Lagrangian
$\{0\}\times{\R^m}^*$. Now define the map $\varphi_t:(\R^n\times{\R^n}^*,\omega)\rightarrow (\R^n\times{\R^n}^*,\omega)$ by
\begin{multline*}
\varphi_t(H_1+V_1,H_2+V_2)=(H_1+V_1,H_2+V_2\\+g_{\gamma(t)}\left[T_{V_1}(\dot\gamma(t))+A_{\dot\gamma(t)}(V_1))-A_{\dot\gamma (t)}(H_1)\right]),
\end{multline*}
where $H_1\in\mathcal{H}$, $H_2\in\mathcal{H}^*$, $V_1\in\mathcal{V}$ and $V_2\in\mathcal{V}^*$. It is easy to prove,  using
the symmetry properties of $T$ and $A$  (see Section \ref{lifts}), that $\varphi_t$ is a
symplectomorphism for every $t\in[a,b]$, and that it is continuous in $t$.
Moreover, if $\tilde{\ell}(t)=\ell_*(t)+\mathcal{V}\times\{0\}$, then
$\varphi_t(\tilde{\ell}(t))=\ell(t)$ and $\varphi_t(L_0)=L_0$, so that applying Lemma~\ref{symplec}
we obtain that  $\mu_{L_0}(\ell)=\mu_{L_0}(\tilde{\ell})$.
As $(\R^n\times{\R^n}^*,\omega)=(\mathcal{H}\times\mathcal{H}^*,\omega_1)\oplus
(\mathcal{V}\times\mathcal{V}^*,\omega_2)$, where $\omega_1$ and $\omega_2$ are the restrictions
of $\omega$, by the additivity property of the Maslov index, $\mu_{L_0}(\tilde{\ell})$ is equal to
the Maslov  index of the path $t\to\tilde{\ell}(t)\cap (\mathcal{H}\times\mathcal{H}^*)=\ell_*(t)$ relatively to
the Lagrangian $L_0\cap(\mathcal{H}\times\mathcal{H}^*)=\{0\}\times\mathcal{H}^*$,
which equals the $\mathcal P$-Maslov index of $x$, and the Maslov index of
\[t\longrightarrow\tilde{\ell}(t)\cap( \mathcal{V}\times\mathcal{V}^*)=\mathcal{V}\times\{0\}\] relatively
to $L_0\cap( \mathcal{V}\times\mathcal{V}^*)=\{0\}\times\mathcal{V}^*$, which is zero. This concludes the proof.
\end{proof}

\begin{subsection}{Index form}\label{sub:indexform} Let us now relate the index form of the horizontal geodesic with the index form of the projected geodesic in the base manifold.
We recall that the index form $I_{\{\gamma,\mathcal{Q}\}}$ is the symmetric bilinear
form on the space of vector fields along the geodesic
obtained as the second variation of the energy functional restricted
to curves departing from $\mathcal{Q}$ and arriving to a fixed point $\gamma(b)\in M$.
It is defined for vector fields along $\gamma$ of Sobolev class $H^1$ that are tangent
to $\mathcal{Q}$ in $a$ and zero in $b$ and it is given by
\begin{align*}
I_{\{\gamma,\mathcal{Q}\}}(E,F)=\int_a^b \left[g(\Ddt E,\Ddt F)+g(R(\dot\gamma,E)\dot\gamma,F)\right]\df
t- g({\mathcal S}^{\mathcal Q}(E(a),\dot\gamma(a)),F(a)).
\end{align*}
Moreover, when the vector fields are of class $C^2$, one can use partial integration
to obtain the following expression for the index form:
\begin{align*}
I_{\{\gamma,\mathcal{Q}\}}(E,F)=\int_a^b \left[g(-\tfrac{\mathrm D^2}{\mathrm dt^2}E\right.&\left.
+R(\dot\gamma,E)\dot\gamma,F)\right]\df t\\&
+g(\Ddt E(a),F(a))- g({\mathcal S}^{\mathcal Q}(E(a),\dot\gamma(a)),F(a)).
\end{align*}
\begin{teo}
Let $\gamma:[a,b]\to M$ be a horizontal geodesic and $\mathcal{P}$ a non-degenerate submanifold through $x(a)$ orthogonal to $x$.  Then
\[I_{\{\gamma,{\mathcal Q})\}}(E,F)=I_{\{x,\mathcal{P}\}}(E_*,F_*)+\int^b_ag(D(E),D(F))\df t
\]
\end{teo}
\begin{proof}
From \eqref{boundaryconditions} we get
\begin{multline*}
g(\Ddt E(a),F(a))-g({\mathcal S}^{\mathcal Q}(E(a),\dot\gamma(a)),F(a))\\
=g(D(E)(a),F(a))+g(\Ddt E_*(a),F_*(a))-g({\mathcal S}^{\mathcal P}(E(a),\dot\gamma(a)),F(a)),
\end{multline*}
for any vector field $F$ along $\gamma$ such that $F(a)\in T_{\gamma(a)}\mathcal Q$.
Having in mind this equation the proof follows the same lines of the proof of \cite[Theorem~3]{One67}.
\end{proof}

The relation between the index forms will be especially interesting in the case when
$M$ and $B$ are \emph{Lorentzian} manifolds, that is, manifolds endowed with a semi-Riemannian
metric of index $1$, and when $\gamma$ is a horizontal \emph{causal} geodesic (that is $g(\dot\gamma,\dot\gamma)\leq 0$),
because in that  case we can use the Lorentzian Morse index theorem (see \cite[Theorem II.5]{PiTau99} and also \cite[Theorems 10.27 and 10.77]{BEE}).
We define
$\omega(\gamma)=\dim \{J\in \J:J(a)=J(b)=0\}$ and $\omega_\delta(\gamma)=
\dim \{J\in \J_\delta (x(a)) :J(a)=J(b)=0\}$ and $\omega_n(\gamma)=\omega(\gamma)-
\omega_\delta(\gamma)$. Moreover, $i(\gamma)$ (resp. $i(x)$) denotes
the index of the index  form restricted to the orthogonal vector
fields to $\gamma$ (resp. to $x$)  vanishing in $a$ and $b$ if $\gamma$ is
timelike (i.e. $g(\dot\gamma,\dot\gamma)<0$) and to the orthogonal vector
fields to $\gamma$ modulo vector fields collinear to $\dot \gamma$ if $\gamma$ is lightlike
(i.e. $g(\dot\gamma,\dot\gamma)=0$).
\begin{teo}\label{indexrelation}
Let $\pi:M\rightarrow B$ be a Lorentzian submersion and let $\gamma:[a,b]\rightarrow M$
be a causal  horizontal geodesic segment. Then
\begin{equation}\label{desigual}i(x)\geq i(\gamma)+\omega_n(\gamma).\end{equation}
\end{teo}
\begin{proof}
Totally analogous to proof of \cite[Theorem 5]{One67}.
\end{proof}
\begin{rem}
We observe that by Theorem \ref{thm:eqMaslovindices} and the equality between the Maslov index of $x$ and
$i(x)$ (cf.~\cite{PiTa02}), the index $i(x)$ coincides with the $\V(a)$-Maslov index
$i_{\V (a)}(\gamma)$, where $\V(a)=\pi^{-1}(x(a))$.   We have
$i(x)=i_{\V (a)}(\gamma)\geq i(\gamma)$, because the index form for conjugate points coincides with the restriction of
the index form for focal points to vector fields vanishing in $a$ and $b$.
Then, from the Lorentzian index theorem, focal points along $\gamma$ occur
before than conjugate points along $x$.
  In the case of Lorentzian submersions the situation is more rigid.  If there exists a conjugate instant  $t_0$ of $\gamma$ that is not $\V (a)$-focal, then  \eqref{desigual} implies that   the total number of  $\V(a)$-focal instants in  $]a,t_0[$ is strictly bigger than the total number of conjugate points along $\gamma$ in $]a,t_0[$.
We observe that, unlike the Riemannian or the causal Lorentzian case, there is in general no obvious relation between
the distribution of conjugate and focal instants along a semi-Riemannian geodesic (see for instance \cite{JavPic08}).
\end{rem}

\end{subsection}
\subsection{Applications}\label{sub:applications}
Recalling Example~\ref{exa:statspacetimesubm}, Proposition~\ref{contribution} and Theorem~\ref{thm:eqMaslovindices}
give us the following information on spacelike geodesics in stationary spacetimes:
\begin{prop}\label{thm:applstationary}
Let $(S\times\R,g)$ be a standard stationary spacetime,
$\gamma=(x,t):[a,b]\rightarrow S\times\R$ a geodesic orthogonal to the flow
lines  of the Killing field $\partial_t$ and $\mathcal L$ the vertical line through $x(a)$. Then
\begin{itemize}
\item[(a)] the $\mathcal L$-focal instants of $\gamma$ in $(S\times\R, g)$
coincide with the conjugate instants of $x$ in $(S,\tilde g)$
and they give the same contribution to the Maslov index
(in particular the number of $\mathcal L$-focal instant of $\gamma$ is finite);
\item[(b)] the $\mathcal L$-focal instants of $\gamma$ are isolated and their contribution to the Maslov index is positive, so that they are always bifurcation points (see \cite{PiPoTa04}).
\end{itemize}
\end{prop}

\begin{rem}
We observe that the $\mathcal L$-focal points of $\gamma$ are in particular pseudo focal points as defined in \cite{JMP07}.
\end{rem}
From Example \ref{exa:kk}, Proposition~\ref{contribution} and Theorems~\ref{thm:eqMaslovindices} and \ref{indexrelation} we obtain the following results for Kaluza-Klein spacetimes.
\begin{prop}\label{thm:applKaluzaKlein} Let $(M,g)$  be a
 Kaluza-Klein spacetime, as in Example \ref{exa:kk}. Consider a geodesic $x:[a,b]\rightarrow  B$ in $(B,h)$  and let
 $\gamma$ be a horizontal lift of $x$ in $(M,g)$.  If $\mathcal L$ is the fiber  through $x(a)$, then
\begin{itemize}
\item[(a)]the $\mathcal L$-focal instants of $\gamma$ in $(M, g)$  coincide with the conjugate instants of
$x$ in $(B,h)$
and they give the same contribution to the Maslov index.
\item[(b)] if $\gamma$ is causal,   then at least one $\mathcal L$-focal point occurs before   than a conjugate point of $\gamma$. In particular,   from \eqref{desigual},  conjugate points in $\gamma$ determine conjugate points in $x$.
\end{itemize}
\end{prop}
We recall that if the contribution to the Maslov index of a focal instant is non null and the instant is isolated and non-degenerate, then it generates bifurcation (see \cite{PiPoTa04}). As a consequence of Proposition~\ref{contribution} and Theorem~\ref{thm:eqMaslovindices} we get the following.
\begin{prop}
Let $\pi:M\rightarrow B$ be a semi-Riemannian submersion with $B$ Riemannian or Lorentzian. If $\gamma$ is a horizontal geodesic such that the projection $x$ is Riemannian or causal, the
${\mathcal Q}$-focal points always generate bifurcation.
\end{prop}
On the other hand, focal points are always isolated in an analytic manifold. As a consequence:
\begin{prop}
Let $\pi:M\rightarrow B$ be a semi-Riemannian submersion with $B$ analytic. Then  a  horizontal geodesic $\gamma:[a,b]\rightarrow M$ admits just a finite number of
$\mathcal Q$-focal points.
\end{prop}
\end{section}

\end{document}